\documentclass[12pt,a4paper]{article}

\usepackage[utf8]{inputenc}     
\usepackage[T2A]{fontenc}       
\usepackage[russian,english]{babel}
\usepackage{amsmath, amssymb}   
\usepackage{graphicx}           
\usepackage{hyperref}           
\usepackage{geometry}           
\usepackage{booktabs}   
\usepackage{tabularx}   
\usepackage{stmaryrd}
\usepackage{mathrsfs}
\usepackage{float}
\usepackage{tikz}
\usetikzlibrary{calc}
\usetikzlibrary{calc}
\usepackage{amsthm}
\newcolumntype{C}{>{\centering\arraybackslash}X}
\def\ternary{\,{\scriptstyle\circ}\,}
\def\cubic{\frak{M}^{\frak c}_2}
\newtheorem{theorem}{Theorem}

\newtheorem{definition}{Definition}
\newtheorem{proposition}{Proposition}
\title{Associative ternary algebras and ternary Lie algebras at cube roots of unity}
\author{
    Anti Maria Aader\textsuperscript{1}, Viktor Abramov\textsuperscript{1}, Olga Liivapuu\textsuperscript{2}\\[6pt]
    \textsuperscript{1}\textit{University of Tartu, Tartu, Estonia}\\
    \textsuperscript{2}\textit{University of Life Sciences, Tartu, Estonia}\\[6pt]
    \texttt{\small anti.maria.aader@ut.ee, viktor.abramov@ut.ee, olga.liivapuu@emu.ee}
}
\date{}

\begin{document}

\maketitle

\begin{abstract}
\noindent
\textbf{Abstract:} We propose an approach to extending the concept of a Lie algebra to ternary structures based on $\omega$-symmetry, 
where $\omega$ is a primitive cube root of unity. 
We give a definition of a corresponding structure, called a ternary Lie algebra at cube roots of unity, or a ternary $\omega$-Lie algebra. 
A method for constructing ternary associative algebras has been developed.
For ternary algebras, the notions of the ternary $\omega$-associator and the ternary $\omega$-commutator are introduced. 
It is shown that if a ternary algebra possesses the property of associativity of the first or second kind, 
then the ternary $\omega$-commutator on this algebra determines the structure of a ternary $\omega$-Lie algebra. 
Ternary algebras of cubic matrices with associative ternary multiplication of the second kind are considered. 
The structure of the 8-dimensional ternary $\omega$-Lie algebra of cubic matrices of the second order is studied, 
and all its subalgebras of dimensions 2 and 3 are determined.
\end{abstract}

\noindent\textbf{Keywords:} Lie algebra; ternary algebra; ternary associativity of the first and second kind; cubic matrices; ternary Lie algebra at cube roots of unity; Pauli Exclusion Principle 

\section{Introduction}
Ternary algebras are algebras with a ternary multiplication law, that is, a multiplication operation which requires three elements to form a product.
Ternary algebras have increasingly become a popular subject of study both in mathematics and in theoretical physics. 
One possible reason for such interest in ternary algebras in theoretical physics may be related to the quark model of elementary particle theory. 
According to the quark model, baryons consist of three quarks; quarks possess three types of color charge, and there exist three generations of quarks. 
These facts may be viewed as the Universe’s hint that the appropriate algebraic structures for describing phenomena at the Planck scale are ternary algebras.
It was precisely these features of the quark model that motivated Nambu in his construction of a generalization of Hamiltonian mechanics~\cite{Nambu(1973)}. 
At the core of this generalization lies a ternary analog of the Poisson bracket. 
Filippov proposed and studied a ternary analog of Lie algebras~\cite{Filippov}, which is now known as a {3-Lie algebra}. 
An essential part of the structure of a 3-Lie algebra is the {Filippov–Jacobi identity}, which is an analog of the Jacobi identity in the theory of Lie algebras. 
Later it was discovered that the Nambu ternary bracket satisfies the Filippov–Jacobi identity, that is, the Nambu ternary bracket defines a 3-Lie algebra structure on the algebra of functions. 
In the 1970s, Bars and Günaydin~\cite{Bars-Gunaydin_1,Bars-Gunaydin_2} proposed a theory of fundamental constituents of matter, which they called {ternons}. 
The ternary algebras played a central role in their theory as the algebraic building blocks for Lie algebras and superalgebras. 
This algebraic part of the theory was based on the work of Kantor~\cite{Kantor}.

At the foundation of many mathematical and physical structures lies the important concept of anti-symmetry. 
For instance, differential forms are antisymmetric functions of vector fields, the Lie bracket and the Poisson bracket are antisymmetric bilinear forms
and the product of generators in the Grassmann algebra is antisymmetric with respect to the permutation of two generators. 
In theoretical physics, the wave function of a quantum system is antisymmetric with respect to fermions, which is the algebraic expression of the Pauli exclusion principle. 
Indeed, a quantum system cannot contain two fermions with identical quantum characteristics, since, by the property of anti-symmetry, the corresponding wave function vanishes in that case. 
In the approaches of Filippov and Nambu to generalizing the concept of Lie and Poisson brackets to algebraic structures with $n$-ary multiplication laws, 
the property of anti-symmetry is preserved, that is, the $n$-ary Filippov and Nambu brackets are antisymmetric. 
Consequently, if two arguments coincide, these brackets vanish. 

However, when passing from binary to ternary structures, one encounters a different type of symmetry, 
which may be called the {rotational symmetry of an equilateral triangle}, or {$\omega$-symmetry}, 
where $\omega$ is a primitive cube root of unity. 
This type of symmetry underlies the structures introduced and studied in~\cite{Abramov-Kerner-LeRoy, Abramov-Kerner-Liivapuu-Shitov,AKL-2020-SPMS,Kerner-Lukierski-2021,Kerner-Lukierski-2022}. 
The essence of this symmetry is that a cyclic permutation of the three arguments of a trilinear form multiplies its value by the primitive cube root of unity $\omega$. 
Let $W$ be a complex vector space and let $\theta$ be a trilinear $W$-valued form on a vector space $V$. 
Then the above symmetry can be expressed as $\theta(x,y,z)=\omega\,\theta(y,z,x)$, where $x,y,z\in V$. 
From this property, it follows that if all three arguments of $\theta$ coincide, that is, $x=y=z$, then the form vanishes. 
However—and this is the fundamental difference from anti-symmetry—if two of the three arguments coincide, the form $\theta$ does not necessarily vanish, 
that is, its value may be nonzero. 
Therefore, a quantum system of several particles whose state is described by a wave function possessing $\omega$-symmetry 
permits the coexistence of two particles with identical quantum characteristics, but forbids the coexistence of three such particles. 
This may be regarded as a ternary generalization of the Pauli exclusion principle, whose physical justification is given in~\cite{Kerner(2017)}.

In the present paper, we develop an approach to extending the concept of a Lie algebra to ternary structures, initiated in~\cite{Abramov_2024}. 
This approach is based on a ternary bracket possessing $\omega$-symmetry. 
We introduce the notion of a {ternary Lie algebra at cube roots of unity or ternary $\omega$-Lie algebra}. 
An important part of this structure is the {$GA(1,5)$-identity}, which we regard as an analog of the Jacobi identity in Lie algebras. 
This identity is a sum of terms that can be obtained through permutations belonging to the general affine group $GA(1,5)$. 
It is well known that an associative (binary) algebra becomes a Lie algebra if it is equipped with the commutator of two elements. 
We develop an analogous method for the construction of ternary $\omega$-Lie algebras. 
To this end, we introduce a {ternary $\omega$-commutator} and show that, in the case of a ternary algebra possessing associativity of the first or second kind, 
the ternary $\omega$-commutator satisfies the $GA(1,5)$-identity. 
Along the way, we also introduce the concept of a {ternary $\omega$-associator}, which unifies all conditions of ternary associativity 
into a single algebraic expression under the assumption that the underlying field is the field of complex numbers. 
In the final section, we investigate the structure of the ternary $\omega$-Lie algebra of cubic matrices of the second order and find all its two-dimensional and three-dimensional subalgebras.

\section{Ternary Lie algebra at cube roots of unity}


The study of ternary algebras, that is, vector spaces with a ternary multiplication law, has always attracted both mathematicians and physicists. Since Lie algebras play an exceptionally important role in algebra, geometry, and theoretical physics, the extension of the concept of Lie algebras to structures with ternary multiplication is of particular interest in this context. Several approaches exist in this area. One of them was proposed by Filippov \cite{Filippov}, and the corresponding structure is called an $n$-Lie algebra. In Filippov's approach, a key role is played by the fact that the Jacobi identity is equivalent to the property in which, in the double Lie bracket, the outer bracket is a derivation of the inner bracket. The second approach \cite{Bremner}  uses the concept of ternary associativity, which is closer to that proposed in this paper. It is important to note the distinction between our approach and the aforementioned ones. As is well known, the binary bracket is skew-symmetric, meaning that the interchange of two arguments in the Lie bracket is accompanied by a multiplication by -1. This property is preserved when the concept of Lie algebra is extended to ternary structures in \cite{Bremner,Filippov}, i.e., the ternary analogue of the Lie bracket is considered to be skew-symmetric. This implies that the ternary bracket does not change under cyclic permutations of its three elements, but changes sign under non-cyclic permutations. In the approach proposed in this paper, we do not require the ternary bracket to be skew-symmetric. Our approach is based on the symmetries of an equilateral triangle, which can be described using the cube roots of unity.

The symmetric groups \( S_3 \) and \( S_5 \), and their subgroups, will play an important role in our approach to extend the concept of Lie algebras to ternary structures. The subgroups of the symmetric group \( S_3 \) are required to describe the symmetry properties of a ternary analogue of the Lie bracket, while the subgroups of the symmetric group \( S_5 \) are necessary for a ternary analogue of the Jacobi identity. Cyclic permutations of three elements form a subgroup of order three in the symmetric group \( S_3 \). This subgroup is called the alternating subgroup of degree 3 and is denoted by \( {\cal A}_3 \). Thus, ${\cal A}_3=\{(1),(1\,2\,3),(1\,3\,2)\}$, where $(1)$ is the identity permutation. The alternating group \( {\cal A}_3 \) is isomorphic to the group \( \mathbb{Z}_3 \) and it is also isomorphic to the multiplicative group of the cube roots of unity $\{1,\omega,\overline\omega\}$, where $\omega$ is a primitive cube root of unity, for example, $\omega=\exp(2\pi i/3)$.

To describe the structure of an identity that we consider as a ternary analog of the Jacobi identity, we need a subgroup of the symmetric group \( S_5 \), which is called the general affine group and is denoted by \( GA(1,5) \). The general affine group of degree 1 over the finite field \( \mathbb{F}_5 \) can be described in several equivalent ways. It is the group of affine transformations of the field \( \mathbb{F}_5 \), that is, the group of all mappings of the form  
\[
x \mapsto a x + b, \quad \text{where } a \in \mathbb{F}_5^\times, \; b \in \mathbb{F}_5.
\]
Additionally, the group \( GA(1,5) \) is isomorphic to the semidirect product of the cyclic group \( \mathbb{Z}_5 \) and the cyclic group \( \mathbb{Z}_4 \):
\[
GA(1,5) \simeq \mathbb{Z}_5 \rtimes \mathbb{Z}_4.
\]

The group \( GA(1,5) \) admits a presentation in terms of two generators \( \sigma \) and \( \tau \) and the following defining relations:
\[
\sigma^5 = e, \quad \tau^4 = e, \quad \tau \sigma \tau^{-1} = \sigma^2,
\]
where \( e \) is the identity element. Here, \( \sigma \) corresponds to the subgroup \( \mathbb{Z}_5 \) (the group of translations), and \( \tau \) corresponds to the subgroup \( \mathbb{Z}_4 \) (the multiplicative group \( \mathbb{F}_5^\times \), acting by scalar multiplication). For our purposes, it is important to consider the representation of the group $GA(1,5)$ by permutations of five elements, that is, its realization as a subgroup of the symmetric group $S_5$. The group $GA(1,5)$ is generated by two cycles 
\begin{equation}
\sigma=(1\;2\;3\;4\;5),\;\;\tau=(2\;4\;5\;3).
\label{cycles of GA(1,5)}
\end{equation}
\noindent
In the context of a ternary analog of the Jacobi identity, the group $GA(1,5)$ plays a crucial role in describing the symmetry structure of the identity. Specifically, the set of permutations appearing in the ternary analog of the Jacobi identity can be identified with the action of $GA(1,5)$ on the indices of the elements involved in the double ternary bracket. This group captures the essential affine symmetries underlying the algebraic identity.

Before giving the definition of a ternary Lie algebra at the cube roots of unity, we will introduce the necessary concepts and notations. Let $\cal L$ be a vector space over the field of complex numbers $\mathbb C$ endowed with a ternary bracket $[-,-,-]:{\cal L}\times {\cal L}\times {\cal L}\to {\cal L}$. In what follows, we assume that this ternary bracket is additive in each argument, that is, $[v_1+v_2,w,u] = [v_1,w,u]+[v_2,w,u]$ and a similar property holds for the second and third arguments. If, in addition, the ternary bracket is $\mathbb C$-homogeneous in each argument, that is, if complex numbers can be factored out of the bracket from any position, we will call the ternary bracket $\mathbb C$-linear. Hence, for a $\mathbb C$-linear ternary bracket and any complex number $a$ we have
$$
[a\,u,v,w]=[u,a\,v,w]=[u,v,a\,w]=a\,[u,v,w].
$$
In the following, we will study examples of ternary brackets that do not have the property of homogeneity with respect to complex numbers but are homogeneous with respect to real numbers. In this case, we will call the ternary bracket \( \mathbb{R} \)-linear. Thus, in the case of \( \mathbb{R} \)-linearity, an additive ternary bracket is defined on a complex vector space $V$, but it does not possess the property of homogeneity with respect to complex numbers; instead, it is homogeneous with respect to real numbers. Of course, if a ternary bracket is homogeneous with respect to complex numbers, then, as a special case, it is also homogeneous with respect to real numbers. However, in this case, we will not use the term \( \mathbb{R} \)-linearity, since homogeneity with respect to real numbers follows as a special case from homogeneity with respect to complex numbers.  

Let \( v_1, v_2, v_3, v_4, v_5 \) be elements of the vector space \( {\cal L} \). The elements of the general affine group \( GA(1,5) \), considered as permutations of five elements, act naturally on the set \( \{1, 2, 3, 4, 5\} \). We introduce the symbol \(
\circlearrowleft\) for the sum of five double ternary brackets obtained by cyclic permutations of the arguments \( v_1, v_2, v_3, v_4, v_5 \) as follows:
$$
\circlearrowleft \big[[v_1,v_2,v_3],v_4,v_5\big]=
    \sum_{k=0}^4\;\big[[v_{\sigma^k(1)},v_{\sigma^k(2)},v_{\sigma^k(3)}],v_{\sigma^k(4)},v_{\sigma^k(5)}\big],
$$
where $\sigma$ is the cyclic permutation \eqref{cycles of GA(1,5)} and $\sigma^0$ is the identity permutation. If the symbol of cyclic permutations of five arguments is placed before parentheses that contain the sum of several double ternary brackets, we mean that the symbol $\circlearrowleft$ is applied to each term separately.
\begin{definition}
A complex vector space ${\cal L}$ with a ternary bracket $[-,-,-]:{\cal L}\times {\cal L}\times {\cal L}\to {\cal L}$ defined on it is said to be a ternary Lie algebra at the cube roots of unity if
\begin{itemize}
\item the ternary bracket is either $\mathbb R$-linear or $\mathbb C$-linear,
\item the cyclic permutation of the three variables of the ternary bracket results in multiplying the entire bracket by a primitive cube root of unity
      \begin{equation} 
         [u,v,w]=\omega\,[v,w,u]=\overline{\omega}\,[w,u,v],\;\;\;u,v,w\in {\cal L},
         \label{omega-symmetry}
      \end{equation}
\item the ternary bracket satisfies the $GA(1,5)$-identity
$$
\circlearrowleft\big(\big[[u,v,w],x,y\big]
    +\big[[u,x,v],y,w\big]+\big[[u,y,x],w,v\big]+
                                    \big[[u,w,y],v,x\big]\big)=0.
$$
\end{itemize}
\label{Def of ternary Lie algebra}
\end{definition}
To simplify the terminology, we will henceforth refer to a ternary Lie algebra at the cube roots of unity as a ternary $\omega$-Lie algebra, where \( \omega \) is a primitive cube root of unity. We will call property \eqref{omega-symmetry} of the ternary bracket $\omega$-symmetry with respect to cyclic permutations of the bracket's variables. Non-cyclic permutations can be used to define a new type of bracket. For example, we can define a new bracket via a reflection-type permutation with respect to the central element, by the formula  
$
\llbracket u,v,w\rrbracket = [w,v,u].
$
This new bracket also exhibits symmetry under cyclic permutations of its arguments, but this symmetry is conjugate to the $\omega$-symmetry \eqref{omega-symmetry}. By this, we mean that the new bracket transforms according to the formula
\begin{equation}
\llbracket u,v,w \rrbracket=\overline{\omega}\,\llbracket v,w,u \rrbracket={\omega}\,\llbracket w,u,v \rrbracket.
\label{conjugate symmetry}
\end{equation}
It is also straightforward to verify that the new bracket satisfies the $GA(1,5)$-identity.

It follows from the $\omega$-symmetry of the ternary bracket that
\begin{equation}
[u_1,u_2,u_3]+[u_2,u_3,u_1]+[u_3,u_1,u_2]=0,\;\;u_1,u_2,u_3\in {\cal L}.
\label{sum of cyclic perm give zero}
\end{equation}
This property indicates an analogy with a Lie bracket. Indeed, let $L$ be a Lie algebra with a Lie bracket 
\( [-,-]: L \times L \to L \). Then
\begin{equation}
[a_1, a_2] + [a_2, a_1] = 0,\;\;\forall a_1,a_2\in L.
\label{skew-symmetry}
\end{equation}
Our approach is based on the following interpretation of this identity. Let \( S_2 = \{(1), (12)\} \) be the symmetric group on two elements. This is a cyclic group of order two. It acts naturally on the indices of the elements in the binary bracket \([a_1, a_2]\); we form the sum of the two resulting brackets under this action and set it equal to zero.  

We extend this idea to the ternary bracket. A cyclic group of order three can be realized via cyclic permutations of three elements — specifically, we take \( \mathcal{A}_3 \) as the cyclic group of order three. Then, we apply this group to the indices of the elements in the ternary bracket \([u_1, u_2, u_3]\), form the sum of the resulting three brackets, and equate it to zero. This yields identity \eqref{sum of cyclic perm give zero}.
Thus, identity (\ref{sum of cyclic perm give zero}) can be seen as a ternary extension of (\ref{skew-symmetry}) to the group of cyclic permutations of three elements.  

It also follows from the $\omega$-symmetry that the ternary bracket vanishes as soon as all three of its arguments are equal: for any $u\in{\cal L}$, it holds that $[u,u,u]=0$. However, in the case where two of its arguments are equal, it does not necessarily vanish. Here, we observe a difference between the properties of the ternary bracket proposed in this paper and those of a ternary bracket in a 3-Lie algebra. In a 3-Lie algebra, the ternary bracket is totally skew-symmetric with respect to its arguments, meaning that it vanishes as soon as any two of its three arguments are equal. Given this property, we could refer to totally skew-symmetric ternary brackets as first-order ternary brackets, whereas a ternary bracket with $\omega$-symmetry could be called a second-order ternary bracket.  

In the theory of Lie algebras, simple Lie algebras play an important role. We can extend the concept of a simple Lie algebra to ternary $\omega$-Lie algebras using the following definition.
\begin{definition}
Let $\cal L$ be a ternary $\omega$-Lie algebra and $\cal I\subset {\cal L}$ be its subspace. Then, $\cal I$ is said to be an ideal of a ternary $\omega$-Lie algebra if, for any $a\in {\cal I}$ and $x,y\in{\cal L}$, it holds $[a,x,y]\in{\cal I}.$ A ternary $\omega$-Lie algebra is said to be simple if it has no non-trivial ideals, that is, it has no ideals other than $\{0\}$ and $\cal L$.
\end{definition}
Finite-dimensional Lie algebras are conveniently studied through their structure constants. In the case of a ternary $\omega$-Lie algebra, we introduce structure constants in a manner analogous to that used in the theory of Lie algebras. Let $\{e_1,e_2,\ldots,e_n\}$ be a basis for the vector space of a ternary $\omega$-Lie algebra $\cal L$. Then
\begin{equation}
[e_i,e_j,e_k]=C^m_{ijk}\,e_m,
\end{equation}
where the complex numbers $C^m_{ijk}$ will be referred to as structure constants of a ternary $\omega$-Lie algebra $\cal L$. The structure constants of a ternary $\omega$-Lie algebra form a \((1,3)\) -tensor, which means that under a basis transformation \( \tilde{e}_i = A^j_i e_j \), the structure constants transform according to the tensor law
\begin{equation}
{\widetilde C}^m_{ijk}=A^p_i\,A^r_j\,A^s_k\,(A^{-1})^m_q\;C^q_{prs}.
\label{tensor law}
\end{equation}
It should be noted that the tensorial transformation law for the structure constants of a ternary $\omega$-Lie algebra depends on the type of linearity of the ternary bracket. If the ternary bracket is $\mathbb{C}$-linear, then the change-of-basis matrix is complex. In the case of an $\mathbb{R}$-linear ternary bracket, only real basis transformations of the ternary $\omega$-Lie algebra should be considered, and in this case formula \eqref{tensor law} defines a tensor representation of a real group, such as the rotation group.

The $\omega$-symmetry of the ternary bracket with respect to cyclic permutations implies the corresponding symmetry of the structure constants with respect to cyclic permutations of the subscripts
\begin{equation}
C^m_{ijk}=\omega\,C^m_{jki}=\overline{\omega}\,C^m_{kij}.
\end{equation}
From the \( GA(1,5) \)-identity, it follows that the structure constants of the ternary $\omega$-Lie algebra satisfy the system of equations
\begin{equation}
\circlearrowleft (C^m_{\underline{i}\,\underline{k}\,\underline{l}}\,C^p_{m\,\underline{r}\,\underline{s}}+C^m_{\underline{i}\,\underline{r}\,\underline{k}}\,C^p_{m\,\underline{s}\,\underline{l}}+C^m_{\underline{i}\,\underline{s}\,\underline{r}}\,C^p_{m\,\underline{l}\,\underline{k}}+C^m_{\underline{i}\,\underline{l}\,\underline{s}}\,C^p_{m\,\underline{k}\,\underline{r}})=0,
\label{identity for structure constants}
\end{equation}
where the underlined five subscripts undergo cyclic permutations, for example
$$
\circlearrowleft C^m_{\underline{i}\,\underline{k}\,\underline{l}}\,C^p_{m\,\underline{r}\,\underline{s}}=C^m_{ikl}\,C^p_{mrs}+C^m_{klr}\,C^p_{msi}+C^m_{lrs}\,C^p_{mik}+C^m_{rsi}\,C^p_{mkl}+C^m_{sik}\,C^p_{mlr}.
$$
We will refer to the system of equations \eqref{identity for structure constants} as the \( GA(1,5) \)-system of equations.
Let \( {\cal T}^{1,3}({\cal L}) \) be the space of complex-valued \((1,3)\)-tensors of the vector space of the ternary omega-Lie algebra \( \cal L \). The tensors possessing omega-symmetry form a subspace in \( {\cal T}^{1,3}({\cal L}) \), which we denote by \( {\cal T}^{1,3}_{\omega}({\cal L}) \). The tensors in the space \( {\cal T}^{1,3}_{\omega}({\cal L}) \) that are solutions of the \( GA(1,5) \)-system of equations are the structure constants of the ternary \( \omega \)-Lie algebra.

Equation \eqref{tensor law} defines a tensor representation of the group \( GL(n, \mathbb{C}) \) (or one of its subgroups) in the space \( {\cal T}^{1,3}_{\omega}({\cal L}) \). Irreducible representations play an important role in this context, as they are directly related to the classification of \( n \)-dimensional ternary \( \omega \)-Lie algebras.
\section{\texorpdfstring{Ternary associativity, ternary $\omega$-associator and ternary $\omega$-commutator}{Ternary associativity, ternary associator and ternary commutator}}
An associative (binary) algebra becomes a Lie algebra if the Lie bracket is defined as the commutator of two elements. In this case, the Jacobi identity is valid due to the associativity of binary multiplication in the original algebra. More precisely, if we expand the double commutators in the Jacobi identity, we obtain a sum of twelve double products, which can be grouped into pairs, each of which forms an associator. Then, by associativity, each such pair vanishes, and the total sum equals zero.  

In this section, we demonstrate that the same scheme can be realized in the case of ternary associativity. We introduce the concepts of the ternary \( \omega \)-associator and the ternary \( \omega \)-commutator and show that, in the case of an associative ternary algebra, the ternary \( \omega \)-commutator satisfies all the requirements of Definition (\ref{Def of ternary Lie algebra}). In particular, it satisfies the \( GA(1,5) \) identity, where the equality to zero is achieved in the same way as in the case of the Jacobi identity: namely, all terms on the left-hand side of the identity are grouped in threes into ternary \( \omega \)-associators, which vanish due to the associativity of ternary multiplication.

Let \( \mathscr A \) be a vector space over complex numbers equipped with a ternary multiplication law 
\(
(u, v, w) \in {\mathscr A} \times {\mathscr A} \times {\mathscr A} \longmapsto u \cdot v \cdot w \in {\mathscr A}.
\)
Here we must make the same remark as in the previous section concerning the linearity of a ternary bracket. When we speak of ternary multiplication, we always assume additivity in each argument, that is
\begin{eqnarray}
&& (u_1+u_2)\cdot v\cdot w=u_1\cdot v\cdot w+u_2\cdot v\cdot w,\nonumber\\
  && u\cdot (v_1+v_2)\cdot w=u\cdot v_1\cdot w+u\cdot v_2\cdot w,\nonumber\\
    && u\cdot v\cdot (w_1+w_2)=u\cdot v\cdot w_1+u\cdot v\cdot w_2.\nonumber
\end{eqnarray}
If, in addition to additivity, the ternary multiplication is $\mathbb{C}$-homogeneous in each argument—that is, complex scalars can be factored out of the ternary product—then we call the ternary multiplication $\mathbb{C}$-linear. Hence, in the case of $\mathbb C$-linear ternary multriplication we have
$$
(\alpha\,u)\cdot v\cdot w=u\cdot (\alpha\,v)\cdot w=u\cdot v\cdot (\alpha\,w)=\alpha\;u\cdot v\cdot w,\;\;\alpha\in{\mathbb C}.
$$
In order to establish a connection with quantum physics, we will consider ternary multiplications that are $\mathbb{C}$-homogeneous in the first and third arguments, but conjugate-homogeneous in the second argument \cite{Bruce,Hestenes}. Such ternary multiplications will be referred to as conjugate-linear. Hence, in the case of conjugate-linear ternary multiplication we have
$$
(\alpha\,u)\cdot v\cdot w=u\cdot v\cdot (\alpha\,w)=\alpha\;u\cdot v\cdot w,\;\;\;
         u\cdot(\alpha\,v)\cdot w=\overline{\alpha}\;u\cdot v\cdot w.
$$
Thus, by a ternary algebra we mean a vector space over the field of complex numbers equipped with a ternary multiplication that is either $\mathbb{C}$-linear or conjugate-linear. 

If the ternary multiplication satisfies the property  
\begin{equation}
(s\cdot u \cdot v) \cdot x \cdot y =s \cdot (u \cdot v \cdot x) \cdot y = s \cdot u \cdot (v \cdot x \cdot y),
\label{associativity I}
\end{equation}  
then the ternary algebra \( \mathscr A \) is called associative of the first kind. If instead the property  
\begin{equation}
(s\cdot u \cdot v) \cdot x \cdot y =s \cdot (x \cdot v \cdot u) \cdot y = s \cdot u \cdot (v \cdot x \cdot y),
\label{associativity II}
\end{equation}  
holds, then \( \mathscr A \) is said to be associative of the second kind. In what follows, we will refer to the ternary algebra \( \mathscr A \) simply as associative if it satisfies one of the above types of associativity, provided that the structure under consideration does not depend on the specific type of ternary associativity. It should be noted that the associativity of the second kind \eqref{associativity II} appears under various names in the literature. A list of these names is provided in \cite{Zapata-Arsiwalla-Beynon}, where the corresponding references are also given. In this work, we use the term "associativity of the second kind" as proposed in \cite{Carlsson}.

One might assume that associativity of the first kind is the fundamental one, while associativity of the second kind is something more specific. However, this is not the case. Ternary structures with associativity of the second kind play an important role in applications. If we do not assume a vector space structure on \( \mathscr A \)—that is, if we treat \( \mathscr A \) simply as a set with a ternary operation satisfying associativity of the second kind—then we arrive at a structure known as a semi-heap. The theory of heaps and semi-heaps was developed in the works of Wagner \cite{Wagner 1,Wagner 2}  in connection with an algebraic approach to the concept of a smooth atlas on a manifold. Moreover, it is possible to construct a ternary multiplication of rectangular matrices such that it satisfies associativity of the second kind \cite{Hestenes}. Finally, ternary multiplications of cubic (i.e., three-dimensional) matrices possessing associativity of the second kind were constructed in \cite{Abramov-Kerner-Liivapuu-Shitov}, where it was also shown that there exists no ternary multiplication of cubic matrices satisfying associativity of the first kind. In the following sections, we will use these structures to construct ternary \( \omega \)-Lie algebras.

Let \( \mathscr A \) be a ternary algebra. We define two quinary operations on \( \mathscr A \) as follows:  
\begin{eqnarray}
&& {\mathfrak Q}^{(1)}_{\omega}(s,u,v,x,y) = (s\cdot u \cdot v) \cdot x \cdot y + \omega\; s\cdot (u \cdot v \cdot x) \cdot y + \overline{\omega}\; s\cdot u \cdot (v \cdot x \cdot y),\nonumber\\
&& {\mathfrak Q}^{(1)}_{\overline{\omega}}(s,u,v,x,y) = (s\cdot u \cdot v) \cdot x \cdot y + \overline{\omega}\; s\cdot (u \cdot v \cdot x) \cdot y + {\omega}\; s\cdot u \cdot (v \cdot x \cdot y).\nonumber
\end{eqnarray}
where \( \omega \) is a primitive cube root of unity and \( \overline{\omega} \) is its complex conjugate. These two quinary operations, \( {\mathfrak Q}_{\omega} \) and \( {\mathfrak Q}_{\overline{\omega}} \), will be referred to as the ternary \( \omega \)-associator of first kind and the ternary \( \overline{\omega} \)-associator of first kind, respectively. Similarly, we define the ternary \( \omega \)- and \( \overline{\omega} \)-associators of the second kind
\begin{eqnarray}
&& {\mathfrak Q}^{(2)}_{\omega}(s,u,v,x,y) = (s\cdot u \cdot v) \cdot x\cdot y + \omega\; s \cdot (x \cdot v \cdot u) \cdot y + \overline{\omega}\;s\cdot u \cdot (v \cdot x \cdot y),\nonumber\\
&& {\mathfrak Q}^{(2)}_{\overline{\omega}}(s,u,v,x,y) = (s\cdot u \cdot v) \cdot x\cdot y + \overline{\omega}\; s \cdot (x \cdot v \cdot u) \cdot y + {\omega}\; s\cdot u \cdot (v \cdot x \cdot y).\nonumber
\end{eqnarray}
From the identity \(1 + \omega + \overline{\omega} = 0\) for the cube roots of unity, it follows that if the ternary multiplication in the algebra \(\mathscr A\) is associative of the first kind, then the ternary \(\omega\)- and \(\overline{\omega}\)-associators of the first kind vanish. The same holds for the case of second-kind associativity and the corresponding \(\omega\)- and \(\overline{\omega}\)-associators of the second kind.  

The converse is also true: if both the ternary \(\omega\)- and \(\overline{\omega}\)-associators vanish for any five elements of the ternary algebra \(\mathscr A\), then the ternary multiplication is associative. We will prove this only for the ternary $\omega$-associators of the first kind; the case of the second kind is treated analogously. The ternary associators of the ternary algebra \(\mathscr A\) are defined as two quinary operations, \( {\mathfrak t}_1 \) and \( {\mathfrak t}_2 \), given by the following formulas \cite{Bremner}:
\begin{eqnarray}
&& {\mathfrak t}_1(s,u,v,x,y)=(s \cdot u \cdot v) \cdot x \cdot y-s \cdot (u \cdot v \cdot x) \cdot y,\nonumber\\
   && {\mathfrak t}_2(s,u,v,x,y)=s \cdot (u \cdot v \cdot x) \cdot y-s \cdot u \cdot (v \cdot x \cdot y).\nonumber
\end{eqnarray}
Obviously, if the ternary associators \( {\mathfrak t}_1 \) and \( {\mathfrak t}_2 \) vanish identically, then \( \mathscr A \) is a ternary algebra with associativity of the first kind. However, it is easy to see that the ternary associators \( {\mathfrak t}_1 \) and \( {\mathfrak t}_2 \) can be expressed in terms of the first-kind \( \omega \)- and \( \overline{\omega} \)-associators as follows:
$$
{\mathfrak t}_1=\frac{{\mathfrak Q}^{(1)}_{\omega}-\omega\,{\mathfrak Q}^{(1)}_{\overline{\omega}}}{1-\omega},\;\;
    {\mathfrak t}_2=\frac{{\mathfrak Q}^{(1)}_{\omega}-{\mathfrak Q}^{(1)}_{\overline{\omega}}}{\omega-\overline{\omega}}.
$$
It follows that the identically vanishing ternary \( \omega \)- and \( \overline{\omega} \)-associators imply the vanishing of the ternary associators \( {\mathfrak t}_1 \) and \( {\mathfrak t}_2 \), which in turn implies associativity of the first kind. Thus, we proved the following statement:
\begin{proposition}
A ternary algebra $\mathscr{A}$ is an associative ternary algebra of the first (second) kind if and only if the ternary $\omega$- and $\overline{\omega}$-associators of the first (second) kind \( {\mathfrak Q}^{(1)}_{\omega}, {\mathfrak Q}^{(1)}_{\overline{\omega}} \) (\( {\mathfrak Q}^{(2)}_{\omega}, {\mathfrak Q}^{(2)}_{\overline{\omega}} \)) vanish identically.
\end{proposition}
We now proceed to construct a ternary $\omega$-Lie algebra based on a ternary associative algebra $\mathscr{A}$ (of the first or second kind). To this end, we define the ternary bracket as follows:
\begin{equation}
[s,u,v] = s\cdot u\cdot v + \omega\; u\cdot v\cdot s + \overline{\omega}\; v\cdot s\cdot u + v\cdot u\cdot s + \overline{\omega}\; u\cdot s\cdot v + \omega\; s\cdot v\cdot u.
\label{omega-bracket}
\end{equation}
It is straightforward to verify that the ternary bracket \eqref{omega-bracket} possesses $\omega$-symmetry; that is, a cyclic permutation of the bracket's arguments $s \to u$, $u \to v$, $v \to s$ results in multiplication of the entire bracket by $\omega$, that is, $[s,u,v]=\omega\,[u,v,s]$. As noted above, if a ternary bracket exhibits $\omega$-symmetry, we may define a new ternary bracket with conjugate symmetry, i.e., $\overline\omega$-symmetry. Indeed, let us define the new ternary bracket as previously indicated, namely, by formula $\llbracket s,u,v\rrbracket=[v,u,s]$. In this case, an explicit expression for the new ternary bracket can be obtained
\begin{equation}
\llbracket s,u,v\rrbracket=s\cdot u\cdot v + \overline\omega\; u\cdot v\cdot s + {\omega}\; v\cdot s\cdot u + v\cdot u\cdot s + {\omega}\; u\cdot s\cdot v + \overline\omega\; s\cdot v\cdot u.
\label{omega-conjugate-bracket}
\end{equation}
Thus, the new ternary bracket can be obtained from bracket \eqref{omega-bracket} by replacing $\omega \leftrightarrow \overline{\omega}$.

Obviously, if the ternary multiplication is $\mathbb{C}$-linear, then the ternary bracket \eqref{omega-bracket} is also $\mathbb{C}$-linear. In the case where the ternary multiplication in $\mathscr{A}$ is conjugate-linear, the ternary bracket is neither $\mathbb{C}$-homogeneous nor conjugate-homogeneous with respect to any of its arguments. The reason is that the ternary bracket is constructed as a linear combination of all permutations of the arguments in the ternary product. As a result of these permutations, a given argument in the ternary bracket \eqref{omega-bracket} shifts its position within the ternary product and thereby loses both homogeneity and conjugate-homogeneity. However, a conjugate-linear ternary multiplication is, in particular, $\mathbb{R}$-linear, and this linearity is preserved in the construction of the ternary commutator \eqref{omega-bracket}. Thus, in the case of a conjugate-linear ternary multiplication, the resulting ternary bracket is neither $\mathbb{C}$-linear nor conjugate-linear, but it inherits the $\mathbb{R}$-linearity of the original ternary multiplication. We will refer to the ternary brackets \eqref{omega-bracket},\eqref{omega-conjugate-bracket} as the ternary $\omega$-bracket and $\overline{\omega}$-bracket, respectively. 

In the case of a binary algebra, the commutator of two elements measures the non-commutativity of the multiplication in the algebra. In other words, if the binary algebra is commutative—i.e., swapping two factors in the product does not change the result—then the commutator of any two elements vanishes. It is natural to require an analogous property for the ternary commutator introduced above \eqref{omega-bracket}. Recall that a ternary multiplication is said to be commutative if its value remains unchanged under any permutation of its arguments. It is easy to see that, in the case of a commutative ternary multiplication, the ternary commutator \eqref{omega-bracket} vanishes for any three elements of the algebra. The condition of commutativity for ternary multiplication is a strong one. It can be weakened by considering various types of ternary commutativity. For instance, if the ternary product remains invariant under the permutation of the first two factors, such a ternary multiplication will be called left-commutative. Accordingly, for a left-commutative ternary multiplication, identity $s\cdot u\cdot v=u\cdot s\cdot v$ holds. In the case of a left-commutative ternary multiplication, the ternary commutator (1) becomes, so to speak, a “truncated” version
\begin{eqnarray}
&& [ s,u,v ] = (1+\overline{\omega})\;s\cdot u\cdot v +(1+ \omega)\;u\cdot v\cdot s +(\omega+ \overline\omega)\;v\cdot s\cdot u\nonumber\\
    &&\qquad\quad\; =-\omega\;s\cdot u\cdot v-\overline\omega\;u\cdot v\cdot s-v\cdot s\cdot u\nonumber.
\end{eqnarray}
In the case of a left-commutative ternary algebra, we shall omit the inessential factor $-1$ and use the truncated form of the ternary $\omega$-commutator
\begin{equation}
[ s,u,v ] = v\cdot s\cdot u + \omega\;s\cdot u\cdot v + \overline\omega\,u\cdot v\cdot s.
\end{equation}
If the ternary multiplication remains unchanged under cyclic permutations of its arguments, such multiplication is called cyclically commutative. Thus, for any three elements of a ternary algebra with cyclically commutative multiplication, the identity
$s \cdot u \cdot v = u \cdot v \cdot s = v \cdot s \cdot u$
holds. In the case of cyclically commutative ternary multiplication, the ternary commutator \eqref{omega-bracket} vanishes identically. It is clear that if a ternary multiplication is both left-commutative and cyclically commutative, then it is commutative.

An important class of Lie algebras consists of those that can be constructed using the commutator defined on an associative algebra. For example, all matrix Lie algebras belong to this class. As follows from the theorem presented in the following, the ternary $\omega$-commutator \eqref{omega-bracket} allows this construction to be extended to associative (of the first or second kind) ternary algebras and ternary $\omega$-Lie algebras.
\begin{theorem}
Let $\mathcal{A}$ be an associative ternary algebra of the first or second kind. Then the ternary $\omega$-commutator \eqref{omega-bracket} satisfies the $GA(1,5)$-identity; that is, the associative ternary algebra $\mathcal{A}$, equipped with the ternary $\omega$-commutator, is a ternary $\omega$-Lie algebra.
\end{theorem}
(See~\cite{Abramov_2024} for a proof.)

As a remark on the proof, the following should be noted. The proof essentially reduces to verifying the validity of the $GA(1,5)$-identity for the ternary $\omega$-commutator \eqref{omega-bracket}. Interestingly, this verification exhibits a complete analogy with the Jacobi identity for the binary commutator. Specifically, when expanding the brackets of the ternary $\omega$-commutator in the left-hand side of the $GA(1,5)$-identity using \eqref{omega-bracket}, all ternary products can be grouped in threes into ternary $\omega$- or $\overline\omega$-associators of the first or second kind. Consequently, due to the ternary associativity of the algebra $\mathcal{A}$, the whole sum vanishes.
\section{Construction of semiheaps and associative ternary algebras}
In this section, we address the question of constructing associative ternary algebras. A ternary multiplication can be constructed from binary ones, that is, by successive application of binary multiplications. However, we are interested in those ternary multiplications that cannot be reduced to a combination of binary multiplications. In this section we propose a general structure that can be used to construct various associative of the first or second kind ternary algebras. 

Let $\mathscr{M}$ be a set, and let $\alpha$ be a mapping that assigns to each ordered pair $(u, v)$ of elements in $\mathscr{M}$ a transformation $\alpha_{u,v}$ of the set $\mathscr{M}$. Since the set of all transformations of $\mathscr{M}$ forms a semigroup under composition, we may regard $\alpha$ as a mapping from the direct product $\mathscr{M} \times \mathscr{M}$ into the semigroup of transformations of $\mathscr{M}$. Define a ternary multiplication on the set $\mathscr{M}$ by
\begin{equation}
u \cdot v \cdot w = \alpha(u,v) w,
\label{general ternary multiplication}
\end{equation}
where $u, v, w \in \mathscr{M}$, $\alpha(u, v):{\mathscr M}\to{\mathscr M}$ and the transformation $\alpha(u, v)$ maps the element $w$ to the element $\alpha(u, v)\, w$. Leaving some freedom of expression, the ternary multiplication process \eqref{general ternary multiplication} can be described as follows: From the first two elements, we construct an operator and then apply this operator to the third element. All ternary multiplications that we will use in the sequel possess this structure.

We will determine the conditions that the mapping $\alpha$ must satisfy for the ternary multiplication \eqref{general ternary multiplication} to be associative of the first or second kind. Note that identity
\begin{equation}
(u \cdot v \cdot w) \cdot s \cdot t = u \cdot v \cdot (w \cdot s \cdot t),
\label{Icondition I}
\end{equation}
must hold in both cases, that is, first-kind and second-kind associativity. The second identity takes the form
\begin{equation}
u \cdot (v \cdot w \cdot s) \cdot t = u \cdot v \cdot (w \cdot s \cdot t),
\label{Icondition II}
\end{equation}
in the case of first-kind associativity, and the form
\begin{equation}
u \cdot (v \cdot w \cdot s) \cdot t = u \cdot s \cdot (w \cdot v \cdot t),
\label{IIcondition II}
\end{equation}
in the case of second-kind associativity. For the mapping $\alpha$, we introduce the transposed mapping defined by $\alpha^T(u, v) = \alpha(v, u)$.
\begin{proposition}
The necessary and sufficient conditions for \eqref{Icondition I}, \eqref{Icondition II}, and \eqref{IIcondition II} are the following:
\begin{eqnarray}
&& \eqref{Icondition I}\;\Leftrightarrow\;\alpha(\alpha(u,v) w,s)=\alpha(u,v)\circ\alpha(w,s),\label{propostion_1_I}\\
&& \eqref{Icondition II}\;\Leftrightarrow\;\alpha(u,\alpha(v,w) s)=\alpha(u,v)\circ\alpha(w,s),\label{propostion_1_II}\\
&& \eqref{IIcondition II}\;\Leftrightarrow\;\alpha(u,\alpha(v,w) s)=\alpha(u,s)\circ\alpha^T(v,w).\label{propostion_1_III}
\end{eqnarray}
\label{Proposition associativity conditions}
\end{proposition}
\begin{proof}
We have
$$
u\cdot v\cdot (w\cdot s\cdot t)=\big(\alpha(u,v)\circ\alpha(w,s)\big) t
$$
and
$$
(u\cdot v\cdot w)\cdot s\cdot t=\big(\alpha(u,v) w \big)\cdot s\cdot t=
       \alpha\big(\alpha(u,v) w,s\big) t.
$$
It follows from this that \eqref{propostion_1_I} holds. Similarly, \eqref{propostion_1_II} and \eqref{propostion_1_III} are proved.
\end{proof}
Thus, it follows from the proven proposition that the set $\mathscr{M}$ with the ternary multiplication \eqref{general ternary multiplication}, where the mapping $\alpha$ satisfies conditions \eqref{propostion_1_I} and \eqref{propostion_1_III}, is a semiheap.

The most well-known example of a semiheap is the set $\mathfrak{P}(A, B)$ of all binary relations between the elements of sets $A$ and $B$. Recall that for two binary relations $R \in \mathfrak{P}(A, B)$ and $S \in \mathfrak{P}(B, C)$, their composition is defined by the formula
$$
R \circ S = \{ (a, c) \in A \times C : \exists\; b \in B,\ (a, b) \in R,\ (b, c) \in S \}.
$$
By its very definition, composition is not a closed binary operation. However, it is possible to construct a closed ternary multiplication of binary relations by combining the composition of relations with their inversion. Recall that
$$
R^{-1} = \{ (b, a) \in B \times A : (a, b) \in R \}.
$$
We now show that the ternary multiplication of binary relations admits the structure of the ternary multiplication \eqref{general ternary multiplication}.
Let $R, S$ be binary relations between elements of the sets $A$ and $B$. Define a mapping
$$
\alpha : \mathfrak{P}(A, B) \times \mathfrak{P}(A, B) \to \mathfrak{P}(A),
$$
where $\mathfrak{P}(A)$ denotes the set of all binary relations on the set $A$, by the formula
$$
\alpha(R, S) = R \circ S^{-1}.
$$
A binary relation on the set $A$, via composition of relations, defines a transformation of the set $\mathfrak{P}(A, B)$, that is,
$$
T \in \mathfrak{P}(A, B) \mapsto V \circ T \in \mathfrak{P}(A, B),
$$
where $V \in \mathfrak{P}(A)$. Following formula \eqref{general ternary multiplication}, we define the ternary multiplication of binary relations between elements of the sets $A$ and $B$ by
$$
R \cdot S \cdot T = \alpha(R, S) \circ T = R \circ S^{-1} \circ T.
$$
The mapping $\alpha$ satisfies conditions \eqref{propostion_1_I} and \eqref{propostion_1_III}, which follows directly from the associativity of the composition of binary relations and the properties of relational inversion.

Since we are interested in associative trilinear ternary operations on sets equipped with the structure of a vector space, that is, in associative ternary algebras, we apply the approach of \eqref{general ternary multiplication} in the case where $\mathscr{M}$ is an additive abelian group, $\alpha: \mathscr{M} \times \mathscr{M} \to \mathscr{A}$, where $\mathscr{A}$ is a unital associative ring, and $\phi: \mathscr{A} \to \text{End}\,\mathscr{M}$ is a representation of the ring $\mathscr{A}$ in the group $\mathscr{M}$. Multiplication in the ring $\mathscr{A}$ will be denoted by the juxtaposition of elements and the endomorphism of the group $\mathscr{M}$ corresponding to an element $a$ of the ring $\mathscr{A}$ will be denoted by $\phi_a$. It is clear that $\mathscr{M}$ is a left $\mathscr{A}$-module if the left action of the ring $\mathscr{A}$ is defined by the formula $a \centerdot u = \phi_a(u)$, where $a \in \mathscr{A}$, $u \in \mathscr{M}$. Note that $\mathscr{M}$ is called a representation module of the ring $\mathscr{A}$. We assume that $\alpha: \mathscr{M} \times \mathscr{M} \to \mathscr{A}$ is an $\mathscr{A}$-valued 2-form, additive in each argument. Thus, we have the sequence of mappings
\begin{equation}
{\mathscr M}\times{\mathscr M}\;\xrightarrow{\hspace{0.2cm}\alpha\hspace{0.2cm}}\;{\mathscr A}\;\xrightarrow{\hspace{0.2cm}\phi\hspace{0.2cm}}\;\text{End}\,{\mathscr M}.
\label{sequence of mappings}
\end{equation}
This sequence may serve as a basis for constructing a ternary multiplication \eqref{general ternary multiplication} on the $\mathscr A$-module $\mathscr M$. The structure of this ternary multiplication can be described as follows: Given three elements $u, v, w$ of the module $\mathscr{M}$, we first associate with the pair $u, v$ an element of the ring $\alpha(u, v)$ ; then, using the representation $\phi$ of the ring $\mathscr{A}$, we map the element of the ring $\alpha(u, v)$ to the endomorphism $\phi_{\alpha(u,v)}$, and apply this endomorphism to the third element $w$, obtaining the result of the ternary product. If we denote this ternary product by $u\cdot v\cdot w$ then
\begin{equation}
u\cdot v\cdot w=\phi_{\alpha(u,v)}(w)=\alpha(u,v)\centerdot w.
\label{ternary product alpha}
\end{equation}
Obviously, the ternary product \eqref{ternary product alpha} is additive in each argument. 

Conditions \eqref{propostion_1_I}, \eqref{propostion_1_II}, and \eqref{propostion_1_III} now take the form
\begin{equation}
\alpha(\alpha(u,v)\centerdot w,s)=\alpha(u,v)\alpha(w,s)= 
\begin{cases}
\alpha(u,\alpha(v,w)\centerdot s), & \text{(I kind associativity)}\\
\alpha(u,\alpha(s,w)\centerdot v), & \text{(II kind associativity)}.
\label{associativity alpha}
\end{cases}
\end{equation}
It is easy to show that a statement analogous to Proposition \ref{Proposition associativity conditions} holds, namely, that from conditions \eqref{associativity alpha}, there follow first-kind or second-kind ternary associativity of the ternary multiplication \eqref{ternary product alpha}. These conditions are sufficient for the associativity of the ternary multiplication \eqref{ternary product alpha} in the case where the left $\mathscr{A}$-module $\mathscr{M}$ is exact, that is, its annihilator consists only of the zero element of the ring, and $\text{Im}\,\alpha = \mathscr{A}$.

In the case of a right module, a ternary multiplication can be defined by a formula analogous to formula \eqref{ternary product alpha}. Let $\mathscr{M}$ be a right $\mathscr{B}$-module, where $\mathscr{B}$ is a unital associative ring, and let $\beta: \mathscr{M} \times \mathscr{M} \to \mathscr{B}$ be an additive $\mathscr{B}$-valued 2-form. The ternary multiplication on $\mathscr{M}$ is defined by
\begin{equation}
u \cdot v \cdot w = u \centerdot \beta(v, w).
\label{ternary product beta}
\end{equation}
In this case, the conditions for ternary associativity take the form
\begin{equation}
\beta(v,w\centerdot\beta(s,t)) = \beta(v,w)\,\beta(s,t)= 
\begin{cases}
\beta(v\centerdot \beta(w,s),t), & \text{(I kind associativity)}\\
\beta(s\centerdot\beta(w,v),t), & \text{(II kind associativity)}.
\label{associativity beta}
\end{cases}
\end{equation}
The described construction of ternary multiplication extends naturally to the case where $\mathscr{A}$ is a unital associative algebra, $\mathscr{M}$ is a left module over the algebra $\mathscr{A}$, and $\alpha$ is a bilinear $\mathscr{A}$-valued form. All vector spaces are assumed to be over the field of complex numbers, and the module structure is assumed to be linear in both arguments. Analogously, the case of a right module over an algebra can be considered. Since $\mathscr{M}$ now has the structure of a complex vector space, equipping it with ternary multiplication yields a ternary algebra. This ternary algebra will be denoted $(\mathscr{M}, \alpha, \mathscr{A})$ in the case of a left module, and $(\mathscr{M}, \beta, \mathscr{B})$ in the case of a right module. Thus, in the ternary algebra $(\mathscr{M}, \alpha, \mathscr{A})$, the multiplication is given by formula (24), and in the ternary algebra $(\mathscr{M}, \beta, \mathscr{B})$, by formula (28). The associativity conditions for these ternary multiplications are given in \eqref{associativity alpha} and \eqref{associativity beta}.

We now present important examples of ternary algebras constructed in the spirit of the ternary algebra $(\mathscr{M}, \alpha, \mathscr{A})$. These are ternary algebras of rectangular matrices, which have been studied in \cite{Carlsson}, \cite{Hestenes}. Let $\mathfrak{M}_{m,n}$ be the vector space of complex $m \times n$ matrices, and let $\mathfrak{M}_m$ be the algebra of complex square matrices of order $m$. It is clear that $\mathfrak{M}_{m,n}$ is a left module over the algebra $\mathfrak{M}_m$, where the left action of $\mathfrak{M}_m$ on $\mathfrak{M}_{m,n}$ is given by the standard matrix multiplication of an $m \times m$ matrix on the left with an $m \times n$ matrix. A ternary multiplication on $\mathfrak{M}_{m,n}$ is defined either by the $\mathfrak M_m$-valued bilinear form $\alpha(X, Y) = X Y^T$, or by the $\mathfrak M_m$-valued conjugate-linear form $\alpha^{\tt c}(X, Y) = X\, {Y}^\dagger$, where $X,Y\in \mathfrak{M}_{m,n}$ and $Y^\dagger=\overline{Y}^T$. In the right-hand sides of these formulas, matrix multiplication, matrix transposition, and complex conjugation are understood. Thus, on the vector space of rectangular matrices $\mathfrak{M}_{m,n}$, we have two ternary algebras, denoted $(\mathfrak{M}_{m,n}, \alpha, \mathfrak{M}_m)$ and $(\mathfrak{M}_{m,n}, \alpha^{\mathrm{c}}, \mathfrak{M}_m)$. In the first algebra, the multiplication is given by the formula
$$
X \cdot Y \cdot Z = X\, Y^T\, Z
$$
and is trilinear, while in the second algebra the multiplication is given by
$$
X \cdot Y \cdot Z = X\, Y^\dagger\, Z
$$
and is conjugate-linear in the second factor. Both ternary algebras are associative of the second kind, which is readily verified by checking conditions \eqref{associativity alpha} corresponding to second-kind associativity.

A second important example of ternary algebras of the form $(\mathscr{M}, \beta, \mathscr{B})$ is constructed using ternary multiplications of cubic matrices. By a cubic matrix, we mean a collection of complex numbers indexed by three natural numbers ranging from 1 to $n$. A cubic matrix will be denoted by $X = (X_{ijk})$, where the indices $i, j, k$ run from 1 to $n$. We refer to $X$ as a cubic matrix of order $n$ because it is often convenient to visualize it as a spatial lattice whose nodes are occupied by the complex numbers $X_{ijk}$. The theory of cubic matrices, and more generally, the theory of $n$-dimensional matrices, is not widely known. Fundamental structures of this theory, such as rank, trace, and determinant, can be found in \cite{Gelfand1}, \cite{Sokolov}. It is clear that the set of cubic matrices of order $n$ forms a vector space over the field of complex numbers under entrywise addition and scalar multiplication. We denote the vector space of cubic matrices of order $n$ by $\mathscr{C}_n$.

For cubic matrices, ternary multiplication is a more natural structure than binary multiplication. Ternary multiplications of cubic matrices, by analogy with the binary multiplication of flat square matrices, can be described as a three-stage composition of linear maps. Let $\mathbb{C}^n$ be the vector space of $n$-dimensional complex vectors, and let $\mathfrak{M}_n$ be the vector space of complex square matrices of order $n$. Then a cubic matrix $Z = (Z_{srk})$ of order $n$ defines a linear map which maps a vector ${a} = (a_k)$ to a matrix $A = (A_{sr})$ via the rule $a=(a_k) \mapsto A=(Z_{srk} \, a_k)$. The next cubic matrix $Y = (Y_{rsp})$ maps the square matrix $A = (A_{sr})$ to a vector $b=(b_p) = (Y_{rsp} A_{sr})$. Thus, at this stage, we obtain a linear map which maps a vector to a vector. To complete this chain and obtain a linear map which maps a vector to a square matrix, we require a third cubic matrix, which, analogously to the first step, maps a vector to a square matrix, i.e., $B=(B_{ij}) = (X_{ijp} \, b_p)$. This process can be represented as a sequence of transformations
\begin{equation}
{\mathbb C}^n\;\xrightarrow{\hspace{0.2cm} Z \hspace{0.2cm}}\;{\mathfrak M}_n\;\xrightarrow{\hspace{0.2cm} Y \hspace{0.2cm}}\;{\mathbb C}^n
       \;\xrightarrow{\hspace{0.2cm} X \hspace{0.2cm}} {\mathfrak M}_n.
\label{sequence of mappings 1}
\end{equation}
Note that at the center of this sequence, that is, in the linear map which maps a matrix to a vector defined by the cubic matrix $Y$, we have an alternative form given by $b_p = Y_{rsp} A_{rs}$. Thus, we obtain the following ternary multiplications of cubic matrices
\begin{eqnarray}
&& (X\cdot Y\cdot Z)_{ijk}=X_{ijp}\,Y_{rsp}\,Z_{srk},\;\;\;(X\cdot Y\cdot Z)_{ijk}=X_{ijp}\,{Y}_{rsp}\,Z_{rsk},\label{cubic matrices product}\\
&& (X\cdot Y\cdot Z)_{ijk}=X_{ijp}\,\overline{Y}_{rsp}\,Z_{srk},\;\;\;(X\cdot Y\cdot Z)_{ijk}=X_{ijp}\,\overline{Y}_{rsp}\,Z_{rsk}\label{cubic matrices product conjugate}
\label{ternary maltiplications cubic}
\end{eqnarray}
where $\overline{Y}_{rsp}$ stands for complex-conjugate of $Y_{rsp}$. Obviously, the ternary products (\ref{cubic matrices product}) are $\mathbb C$-linear and the ternary products (\ref{cubic matrices product conjugate}) are their conjugate-linear counterparts. 
Now our goal is to show that the ternary products (\ref{cubic matrices product}),(\ref{cubic matrices product conjugate}) are associative of the second kind and define ternary algebras of the type \((\mathscr M,\beta,\mathscr B)\).

A cubic matrix can be conveniently represented as a collection of square matrices of order $n$. Indeed, by fixing one of the three indices in the cubic matrix $X = (X_{ijk})$, for example, the last index $k$, we obtain a square matrix of order $n$ whose entries are indexed by $i$ and $j$. Denote the resulting square matrix by $X_{(k)}$. The cubic matrix $X$ can then be identified with the ordered set of $n$ square matrices $(X_{(1)}, X_{(2)}, \ldots, X_{(n)})$, which we denote by $\vec{X}$. Thus, $\vec{X} = (X_{(1)}, X_{(2)}, \ldots, X_{(n)})$ can be viewed as the analogue of an $n$-dimensional vector whose coordinates are square matrices of order $n$. Analogously $\vec{X}^\dagger = (X^\dagger_{(1)}, X^\dagger_{(2)}, \ldots, X^\dagger_{(n)})$.
We will refer to a cubic matrix written in the form $\vec{X}$ as a matrix vector. It is worth noting that the vector space of cubic matrices $\mathscr{C}_n$ is a right module over the algebra of square matrices $\mathfrak{M}_n$. Indeed, given a matrix vector $\vec{X}$, the right action of a square matrix $A$ on $\vec{X}$ is denoted by $\vec{X} \triangleright A$ and is defined as the matrix vector whose $r$th coordinate is the square matrix $X_{(k)} A_{kr}$ ((with summation over $k$ understood)).

Let $Y,Z$ be cubic matrices, written in vector form as $\vec{Y},\vec{Z}$. We now construct a square matrix as follows: at the intersection of the $p$th row and the $k$th column of this matrix stands the trace of the product of the matrices $Y_{(p)}$ and $Z_{(k)}$. The resulting square matrix is denoted by $\text{Tr}(\vec{Y} \vec{Z})$. Thus, we have a bilinear mapping $\beta: \mathscr{C}_n \times \mathscr{C}_n \to \mathfrak{M}_n$, where $\beta(\vec{Y}, \vec{Z}) = \text{Tr}(\vec{Y} \vec{Z})$. Now the ternary multiplications in \eqref{ternary maltiplications cubic} can be written in the form
\begin{equation}
X\cdot Y\cdot Z=\vec{X}\triangleright \text{Tr}\,(\vec{Y}\vec{Z}),\;\;\;
     X\cdot Y\cdot Z=\vec{X}\triangleright \text{Tr}\,(\vec{Y}^\dagger\vec{Z}).
     \label{ternary multiplications cubic 1}
\end{equation}
On the right-hand side of these formulas, the second factor is a matrix. We can compute the trace of this matrix, obtaining a complex number. In this way, we obtain two additional ternary multiplications of cubic matrices, which take the form
\begin{equation}
X\cdot Y\cdot Z=\vec{X}\, \text{Tr}\big(\text{Tr}\,(\vec{Y}\vec{Z})\big),\;\;\;
     X\cdot Y\cdot Z=\vec{X}\,\text{Tr}\big(\text{Tr}\,(\vec{Y}^\dagger\vec{Z})\big).
\label{two additional products}
\end{equation}
In this case, the bilinear form $\beta(Y, Z) = \text{Tr}\big(\text{Tr}(\vec{Y} \vec{Z})\big)$ or $\beta(Y,Z)=\text{Tr}\big(\text{Tr}(\vec{Y}^\dagger \vec{Z})\big)$ is a $\mathbb{C}$-valued form, and in the right-hand sides of formulas \eqref{two additional products}, it is understood that the cubic matrix $X$ is multiplied by the corresponding scalar. Note that the first form is symmetric and bilinear, while the second is conjugate-linear in the first argument and linear in the second.

Having clarified the structure of the ternary products of cubic matrices, we can readily prove their second-kind associativity. We will do this for the first ternary multiplication given in formula \eqref{ternary multiplications cubic 1}; the proof for the second is analogous. To establish second-kind associativity, we show that the bilinear form $\beta$ satisfies conditions \eqref{associativity beta} for the case of second-kind associativity. Thus, we must verify that for arbitrary cubic matrices $\vec{X}, \vec{Y}, \vec{Z}, \vec{V}$, the following identities hold:
\begin{equation}
\beta(\vec{X}, \vec{Y} \triangleright \beta(\vec{Z}, \vec{V})) = \beta(\vec{X}, \vec{Y}) \, \beta(\vec{Z}, \vec{V}),
\label{beta condition fore cubic1}
\end{equation}
and
\begin{equation}
\beta(\vec{X} \triangleright \beta(\vec{Z}, \vec{V}), \vec{Y}) = \beta(\vec{V}, \vec{Z}) \, \beta(\vec{X}, \vec{Y}).
\label{beta condition for cubic 2}
\end{equation}
The first identity follows immediately from the fact that the trace is a linear function. Indeed, on the left hand side of \eqref{beta condition fore cubic1} in the second argument of the form $\beta$, we have the matrix vector $\vec{Y}$, which undergoes a linear transformation by the square matrix $\beta(\vec{Z}, \vec{V})$. By the linearity of the trace, the transformation matrix can be factored out to the right of the form $\beta$, preserving the order of matrix multiplication and thereby proving identity \eqref{beta condition fore cubic1}. The second identity follows from the linearity of the trace and the elementary identity $\text{Tr}(Z_{(i)} V_{(j)}) = \text{Tr}(V_{(j)} Z_{(i)})$. This identity implies that permuting the matrix vectors in $\text{Tr}(\vec{Z}, \vec{V})$ results in transposing the resulting matrix, that is,
$$
(\text{Tr}(\vec{Z}, \vec{V}))^T = \text{Tr}(\vec{V}, \vec{Z}).
$$
This means that, by factoring the linear transformation $\beta(\vec{Z}, \vec{V})$ out of the form $\beta$ to the left, thanks to the linearity of the trace, and simultaneously interchanging the matrix vectors $\vec{Z}, \vec{V}$, we obtain the matrix product on the right-hand side of identity \eqref{beta condition for cubic 2}. The proof of these properties for the forms $\beta$ appearing on the right-hand sides of formulas \eqref{two additional products} is even simpler, since in this case we are dealing with the multiplication of a cubic matrix by a complex scalar. Therefore, properties \eqref{beta condition fore cubic1} and \eqref{beta condition for cubic 2} follow directly from the linearity and conjugate-linearity of the forms.
\section{\texorpdfstring{Ternary $\omega$-Lie algebras of rectangular and cubic matrices}{Ternary Lie algebras of rectangular and cubic matrices}}
In this section, we study the structure of ternary $\omega$-Lie algebras whose ternary $\omega$-commutator is constructed by means of the associative ternary multiplications described in the previous section. We examine in detail the structure of the ternary $\omega$-Lie algebra of cubic matrices of the second order. We will also use the following classification of 2-dimensional ternary $\omega$-Lie algebras \cite{Abramov_2024}:
\begin{theorem}
If $\cal F$ is a two-dimensional ternary $\omega$-Lie algebra, then it is isomorphic to one of the four two-dimensional ternary $\omega$-Lie algebras given by their structure constants in the \mbox{following table:}
\begin{table}[H]
\begin{tabularx}{\textwidth}{CCCCC}
\toprule
\textbf{{Number}} 
	& \textbf{$C^1_{121}$}	& \textbf{$C^2_{121}$} & \textbf{$C^1_{212}$}  & \textbf{$C^2_{212}$}\\
\midrule
{{I}}                  & 0		                          & 0        &     0              &     0             \\

{{II}}	                & 0			                       & 1        &     1              &     0              \\

{{III}}	            & 0		                       & 1        &   0                &     0              \\

{{IV}}	            & 1		                       & 0        &   0                &     $-$1              \\
\bottomrule
\end{tabularx}
\end{table}
\label{theorem classification}
\end{theorem}

We begin with the associative ternary algebra of rectangular matrices $\mathfrak{M}_{m,n}$. The vector space $\mathfrak{M}_{m,n}$ is a module over the algebra of square matrices $\mathfrak{M}_m$ of order $m$, and the associative ternary multiplication is constructed either via the bilinear form $\alpha: \mathfrak{M}_{m,n} \times \mathfrak{M}_{m,n} \to \mathfrak{M}_m$, or via the conjugate-linear form $h: \mathfrak{M}_{m,n} \times \mathfrak{M}_{m,n} \to \mathfrak{M}_m$, where
$$
\alpha(X, Y) = X Y^T, \quad h(X, Y) = X Y^\dagger.
$$
The corresponding ternary multiplications on rectangular matrices take the form
$$
X \cdot Y \cdot Z = \alpha(X, Y)\, Z = X Y^T Z, \quad X \cdot Y \cdot Z = h(X, Y)\, Z = X Y^\dagger Z.
$$
Consider the special case of the associative ternary algebra $\mathfrak{M}_{m,n}$ with $m = 1$. In this case, the elements of the algebra are row matrices, and the vector space of row matrices can be identified with the vector space of complex $n$-dimensional vectors $\mathbb{C}^n$. The bilinear form $\alpha$ then takes the form 
\begin{equation}
\alpha(X, Y) =\sum_i X_i Y_i,
\label{bilinear-form-alpha}
\end{equation}
while the conjugate-linear form $h$ becomes the Hermitian inner product of $n$-dimensional complex vectors, that is,
$$
h(X, Y) =\sum_i X_i \overline{Y}_i,
$$
with $X = (X_1, X_2, \ldots, X_n)$, $Y = (Y_1, Y_2, \ldots, Y_n)$. Henceforth, the associative ternary algebra of $n$-dimensional vectors with ternary multiplication defined via $\alpha$ will be denoted $(\mathbb{C}^n, \alpha)$, and the ternary algebra defined via the Hermitian inner product will be denoted $(\mathbb{C}^n, h)$. The associative ternary algebra $(\mathbb{C}^n, h)$ can be employed in quantum theory, where it provides a framework for treating state vectors and operators as indistinguishable objects. In this setting, ternary multiplication replaces the action of an operator on a state vector by a linear combination of state triplets combined through the ternary product (see the discussion following formula \eqref{general ternary multiplication} and \cite{Bruce}). Following the terminology introduced in \cite{Bruce}, we will refer to the associative ternary algebras $(\mathbb{C}^n, \alpha)$ and $(\mathbb{C}^n, h)$ as vector ternary algebras. 

The vector ternary algebra $(\mathbb C^n,\alpha)$ is left-commutative. Therefore, to construct a ternary $\omega$-Lie algebra we will use a ``reduced'' form of the ternary commutator
\begin{equation}
[X,Y,Z]=\alpha(Z,X)\,Y+\omega\;\alpha(X,Y)\,Z+\overline\omega\;\alpha(Y,Z)\,X,
\label{reduced-commutator-vectors}
\end{equation}
where $X,Y,Z\in {\mathbb C}^n$. If we assume that $\alpha$ is an arbitrary nondegenerate symmetric bilinear form, then we do not obtain a broader class 
of ternary $\omega$-Lie algebras. Indeed, in this case, the ternary bracket (\ref{reduced-commutator-vectors}) is $\mathbb C$-linear and we can use linear complex transformations of the basis of the space $\mathbb C^n$ and by such a transformation a nondegenerate symmetric bilinear form can be brought into the form (\ref{bilinear-form-alpha}). Thus we have only one ternary $\omega$-Lie algebra of this type, which will be referred to as a vector ternary $\omega$-Lie algebra. The structure constants of the ternary $\omega$-Lie algebra take the simplest form in 
the canonical basis $e_1,e_2,\ldots,e_n$, where the $i$-th coordinate of the $n$-dimensional vector $e_i$ is equal to $1$ and all others are $0$. In this case we have 
$\alpha(e_i,e_j)=\delta_{ij}$ and 
\begin{equation}
C^m_{ijk}=\delta_{ik}\delta^m_j+\omega\;\delta_{ij}\delta^m_k+\overline\omega\;\delta_{jk}\delta^m_i.
\end{equation}
\begin{proposition}
For any pair of distinct indices $1\leq i,j\leq n$ the subspace spanned by the vectors $e_i,e_j$ is a subalgebra of the vector ternary $\omega$-Lie algebra and this subalgebra is isomorphic to the II-type algebra in the classification given in Theorem \ref{theorem classification}, that is, the ternary commutation relations of this subalgebra have the form
$$
[e_i,e_j,e_i]=e_j,\;\;[e_j,e_i,e_j]=e_i.
$$
Moreover, for any triple of distinct indices $1\leq i,j,k\leq n$ we have only trivial ternary relations, that is, $[e_i,e_j,e_k]=0$. For any sequence of $r$ integers $1\leq i_1<i_2<\ldots<i_r\leq n$ the subspace spanned by 
the vectors $e_{i_1},e_{i_2},\ldots,e_{i_r}$ is a subalgebra of the vector ternary $\omega$-Lie algebra. 
\end{proposition}

Associative ternary algebras of cubic matrices with ternary multiplications (\ref{cubic matrices product}), (\ref{cubic matrices product conjugate}) provide a possibility for constructing a wide class of ternary $\omega$-Lie algebras.
In this section, we study the structure of the ternary $\omega$-Lie algebra of cubic matrices of the second order. As an associative ternary multiplication of cubic 
matrices we will use the ternary multiplication containing the trace of a square matrix, namely
\begin{equation}
(X\cdot Y\cdot Z)_{ijk}=X_{ijp}\,Y_{rsp}\,Z_{srk}=\vec{X}\triangleright \mbox{Tr}(\vec{Y}\vec{Z}).
\end{equation}
Fixing one of the indices of the cubic matrix $X=(X_{ijk})$, for example $k$, we obtain a quantity with two indices, which we shall consider as a square
matrix. Computing the trace of this square matrix for different values of the index $k$, we obtain the vector $(X_{111}+X_{221},X_{112}+X_{222})$. We shall call this vector 
the trace of the cubic matrix $X$ with respect to the first two indices and denote it by $\mbox{Tr}_{1,2}(X)$. Similarly, the trace of a cubic matrix with respect to the first and third indices
is denoted by $\mbox{Tr}_{1,3}(X)$, and with respect to the second and third indices by $\mbox{Tr}_{2,3}(X)$. From the structure of the ternary multiplication of cubic matrices, the following statement follows:
\begin{proposition}
If the trace of each of the three cubic matrices $X,Y,Z$ of order $r$ with respect to the first two indices is zero, that is, $\mbox{Tr}_{1,2}(X)=\mbox{Tr}_{1,2}(Y)=\mbox{Tr}_{1,2}(Z)=0$, then 
the trace of the ternary $\omega$-commutator with respect to the first two indices is also zero, that is, $\mbox{Tr}_{1,2}([X,Y,Z])=0$. Thus, the cubic matrices of order $r$
having zero trace with respect to the first two indices form a subalgebra of the ternary $\omega$-Lie algebra of all cubic matrices of order $r$.
\end{proposition}
A cubic matrix $X\cdot Y\cdot Z$, which is the product of three cubic matrices $X,Y,Z$, inherits the first two indices from the first matrix $X$. The ternary 
$\omega$-commutator is a linear combination of ternary products of cubic matrices, where the matrices $X,Y,Z$ successively occupy the first position. Since each of the cubic matrices 
has zero trace with respect to the first two indices, we arrive at the statement of the proposition.
\begin{proposition}
If the trace of each of the three cubic matrices $X,Y,Z$ of order $r$ with respect to every pair of indices is zero, that is, $\mbox{Tr}_{1,2}(X)=\mbox{Tr}_{1,3}(X)=\mbox{Tr}_{2,3}(X)=0$
and the same conditions hold for the cubic matrices $Y,Z$, then the trace of the ternary $\omega$-commutator $[X,Y,Z]$ with respect to any pair of indices is also zero. Thus, 
the cubic matrices with trace equal to zero with respect to any pair of indices form a subalgebra of the ternary $\omega$-Lie algebra of cubic matrices with zero trace with respect to the first two indices.
\end{proposition}
The proof of this statement follows from the structure of the ternary multiplication of cubic matrices and the structure of the ternary $\omega$-commutator. 

Let $\cubic$ be the ternary $\omega$-Lie algebra of cubic matrices of the second order. The dimension of the vector space of this algebra is $8$, that is, $\mbox{dim}\,\cubic=8$. 
The subalgebra of cubic matrices with zero trace with respect to the first two indices has dimension $6$. We denote this subalgebra by $\frak T_0$. Thus, $\frak T_0\subset\cubic$  
and $\mbox{dim}\,\frak T_0=6$. The subspace spanned by cubic matrices of the second order with zero trace with respect to any 
pair of indices is a subalgebra of $\frak T_0$. We denote this subalgebra by $\frak T_1$, and its dimension is $2$. Thus, we have a sequence of subalgebras $\frak T_1\subset \frak T_0\subset 
\cubic$, whose dimensions are respectively $2,6,8$. In the space of cubic matrices of the second order we choose a basis in such a way that the first two matrices 
$G_1,G_2$ generate the $2$-dimensional subalgebra of cubic matrices with zero trace with respect to any pair of indices $\frak T_1$. These matrices have the form 
$G_1=-\frac{i\sqrt{2}}{4}F_1,\; G_2=-\frac{i\sqrt{2}}{4}F_2$, where
\vskip.5cm

\begin{tikzpicture}[line join=round,line cap=round,>=stealth]

\newcommand{\drawcube}[9]{%
  \def\s{2.8}      
  \def\dx{1.1}     
  \def\dy{0.9}     

  \coordinate (A) at (0,0);           
  \coordinate (B) at (\s,0);          
  \coordinate (C) at (\s,\s);         
  \coordinate (D) at (0,\s);          

  \coordinate (A') at ($(A)+(\dx,\dy)$);
  \coordinate (B') at ($(B)+(\dx,\dy)$);
  \coordinate (C') at ($(C)+(\dx,\dy)$);
  \coordinate (D') at ($(D)+(\dx,\dy)$);

  \draw[thick] (A)--(B)--(C)--(D)--cycle;
  \draw[dashed] (A')--(B');
  \draw (B')--(C');
  \draw (C')--(D');
  \draw[dashed] (D')--(A');
  \draw[dashed] (A)--(A'); 
  \draw (B)--(B');
  \draw (C)--(C');
  \draw (D)--(D');

  \node[anchor=north east] at (D) {\large #1};            
  \node[anchor=south west, yshift=-11pt] at (C) {\large #2}; 
  \node[anchor=north east] at (A) {\large #3};            
  \node[anchor=north west] at (B) {\large #4};            
  \node[anchor=north east, yshift=14pt] at (D') {\large #5};
  \node[anchor=south west] at (C') {\large #6};
  \node[anchor=north east, yshift=14pt] at (A') {\large #7};
  \node[anchor=north west, yshift=14pt] at (B') {\large #8};
}

\begin{scope}[shift={(0,0)}]
  \drawcube{1}{0}{0}{$-1$}{0}{$-1$}{$-1$}{0}{}
  \node at (1.5,-0.8) {\Large $F_1$};
\end{scope}

\begin{scope}[shift={(6.2,0)}]
  \drawcube{0}{$-1$}{$-1$}{0}{$-1$}{0}{0}{1}{}
  \node at (1.5,-0.8) {\Large $F_2$};
\end{scope}
\end{tikzpicture}

\noindent
It should be noted that in the figure, the faces of the cubic matrix parallel to the plane of the page are ordered according to the third index. That is, the square matrix outlined with bold lines has the elements $(F_1)_{111}=1, (F_1)_{121}=0$ (first row), $(F_1)_{211}=0, (F_1)_{221}=-1$ (second row). The face parallel to it contains the elements of the cubic matrix with $2$ as the third index.

The following four cubic matrices $G_3,G_4,G_5,G_6$ complete the basis $G_1,G_2$ of the subspace $\frak T_1$ to a basis of the subspace $\frak T_0$. These cubic matrices have the form

\vskip.3cm
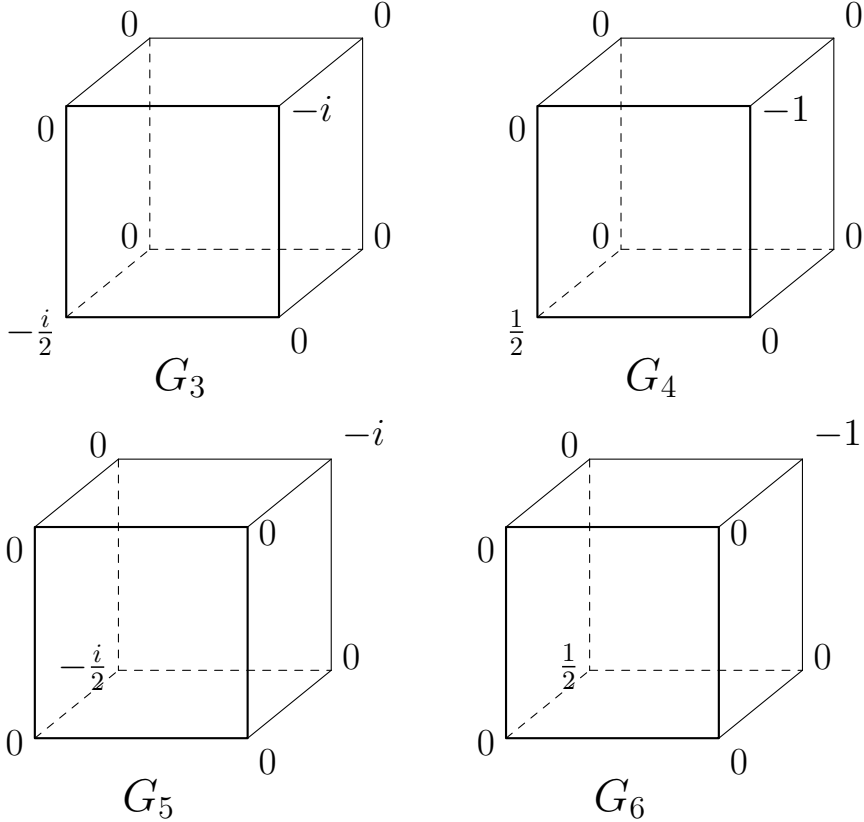
\begin{tikzpicture}[line join=round,line cap=round,>=stealth]

\newcommand{\drawcube}[9]{%
  \def\s{2.8}      
  \def\dx{1.1}     
  \def\dy{0.9}     

  \coordinate (A) at (0,0);           
  \coordinate (B) at (\s,0);          
  \coordinate (C) at (\s,\s);         
  \coordinate (D) at (0,\s);          

  \coordinate (A') at ($(A)+(\dx,\dy)$);
  \coordinate (B') at ($(B)+(\dx,\dy)$);
  \coordinate (C') at ($(C)+(\dx,\dy)$);
  \coordinate (D') at ($(D)+(\dx,\dy)$);

  \draw[thick] (A)--(B)--(C)--(D)--cycle;
  \draw[dashed] (A')--(B');
  \draw (B')--(C');
  \draw (C')--(D');
  \draw[dashed] (D')--(A');
  \draw[dashed] (A)--(A'); 
  \draw (B)--(B');
  \draw (C)--(C');
  \draw (D)--(D');

  \node[anchor=north east] at (D) {\large #1};            
  \node[anchor=south west, yshift=-11pt] at (C) {\large #2}; 
  \node[anchor=north east, yshift=7pt] at (A) {\large #3};            
  \node[anchor=north west] at (B) {\large #4};            
  \node[anchor=north east, yshift=14pt] at (D') {\large #5};
  \node[anchor=south west] at (C') {\large #6};
  \node[anchor=north east, yshift=14pt] at (A') {\large #7};
  \node[anchor=north west, yshift=14pt] at (B') {\large #8};
}

\begin{scope}[shift={(0,0)}]
  \drawcube{0}{$-i$}{$-\frac{i}{2}$}{0}{0}{0}{0}{0}{}
  \node at (1.5,-0.8) {\Large $G_3$};
\end{scope}

\begin{scope}[shift={(6.2,0)}]
  \drawcube{0}{$-1$}{$\frac{1}{2}$}{0}{0}{0}{0}{0}{}
  \node at (1.5,-0.8) {\Large $G_4$};
\end{scope}
\end{tikzpicture}

\begin{tikzpicture}[line join=round,line cap=round,>=stealth]

\newcommand{\drawcube}[9]{%
  \def\s{2.8}      
  \def\dx{1.1}     
  \def\dy{0.9}     

  \coordinate (A) at (0,0);           
  \coordinate (B) at (\s,0);          
  \coordinate (C) at (\s,\s);         
  \coordinate (D) at (0,\s);          

  \coordinate (A') at ($(A)+(\dx,\dy)$);
  \coordinate (B') at ($(B)+(\dx,\dy)$);
  \coordinate (C') at ($(C)+(\dx,\dy)$);
  \coordinate (D') at ($(D)+(\dx,\dy)$);

  \draw[thick] (A)--(B)--(C)--(D)--cycle;
  \draw[dashed] (A')--(B');
  \draw (B')--(C');
  \draw (C')--(D');
  \draw[dashed] (D')--(A');
  \draw[dashed] (A)--(A'); 
  \draw (B)--(B');
  \draw (C)--(C');
  \draw (D)--(D');

  \node[anchor=north east] at (D) {\large #1};            
  \node[anchor=south west, yshift=-11pt] at (C) {\large #2}; 
  \node[anchor=north east, yshift=7pt] at (A) {\large #3};            
  \node[anchor=north west] at (B) {\large #4};            
  \node[anchor=north east, yshift=14pt] at (D') {\large #5};
  \node[anchor=south west] at (C') {\large #6};
  \node[anchor=north east, yshift=14pt] at (A') {\large #7};
  \node[anchor=north west, yshift=14pt] at (B') {\large #8};
}

\begin{scope}[shift={(0,0)}]
  \drawcube{0}{0}{0}{0}{0}{$-i$}{$-\frac{i}{2}$}{0}{}
  \node at (1.5,-0.8) {\Large $G_5$};
\end{scope}

\begin{scope}[shift={(6.2,0)}]
  \drawcube{0}{0}{0}{0}{0}{$-1$}{$\frac{1}{2}$}{0}{}
  \node at (1.5,-0.8) {\Large $G_6$};
\end{scope}
\end{tikzpicture}

The last two cubic matrices $G_7, G_8$ complete the basis of the six-dimensional subspace $\frak T_0$ to a basis of the entire space of all cubic matrices of the second order $\cubic$.

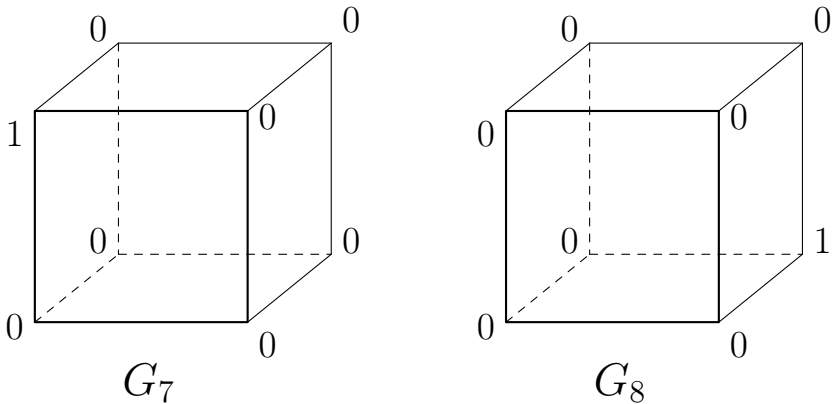
\begin{tikzpicture}[line join=round,line cap=round,>=stealth]

\newcommand{\drawcube}[9]{%
  \def\s{2.8}      
  \def\dx{1.1}     
  \def\dy{0.9}     

  \coordinate (A) at (0,0);           
  \coordinate (B) at (\s,0);          
  \coordinate (C) at (\s,\s);         
  \coordinate (D) at (0,\s);          

  \coordinate (A') at ($(A)+(\dx,\dy)$);
  \coordinate (B') at ($(B)+(\dx,\dy)$);
  \coordinate (C') at ($(C)+(\dx,\dy)$);
  \coordinate (D') at ($(D)+(\dx,\dy)$);

  \draw[thick] (A)--(B)--(C)--(D)--cycle;
  \draw[dashed] (A')--(B');
  \draw (B')--(C');
  \draw (C')--(D');
  \draw[dashed] (D')--(A');
  \draw[dashed] (A)--(A'); 
  \draw (B)--(B');
  \draw (C)--(C');
  \draw (D)--(D');

  \node[anchor=north east] at (D) {\large #1};            
  \node[anchor=south west, yshift=-11pt] at (C) {\large #2}; 
  \node[anchor=north east, yshift=7pt] at (A) {\large #3};            
  \node[anchor=north west] at (B) {\large #4};            
  \node[anchor=north east, yshift=14pt] at (D') {\large #5};
  \node[anchor=south west] at (C') {\large #6};
  \node[anchor=north east, yshift=14pt] at (A') {\large #7};
  \node[anchor=north west, yshift=14pt] at (B') {\large #8};
}

\begin{scope}[shift={(0,0)}]
  \drawcube{1}{0}{0}{0}{0}{0}{0}{0}{}
  \node at (1.5,-0.8) {\Large $G_7$};
\end{scope}

\begin{scope}[shift={(6.2,0)}]
  \drawcube{0}{0}{0}{0}{0}{0}{0}{1}{}
  \node at (1.5,-0.8) {\Large $G_8$};
\end{scope}
\end{tikzpicture}

Let $G_{i_1},G_{i_2},\ldots,G_{i_k}$, where $1\leq i_1<i_2<\ldots<i_k\leq 8$, be an ordered set of cubic matrices.
The subspace of the space $\cubic$ spanned by the cubic matrices $G_{i_1},G_{i_2},\ldots,G_{i_k}$ will be denoted by $\langle G_{i_1},G_{i_2},\ldots,G_{i_k}\rangle$. The following theorem describes all 2-dimensional subalgebras of the ternary $\omega$-Lie algebra of cubic matrices of the second order.
\begin{theorem}
2-dimensional subspaces
$$
\langle G_3,G_6 \rangle,\;\;\langle G_3,G_8 \rangle,\;\;\langle G_4,G_5 \rangle,\;\;\langle G_4,G_8 \rangle,\;\;
     \langle G_5,G_7 \rangle,\;\;\langle G_6,G_7 \rangle,\;\;\langle G_7,G_8 \rangle
$$
are Abelian subalgebras of the ternary $\omega$-Lie algebra $\cubic.$ 2-dimensional subspaces
$$
\langle G_1,G_2 \rangle,\; \langle G_3,G_4\rangle,\;\langle G_3,G_5 \rangle,\;\langle G_3,G_7 \rangle,\;
    \langle G_4,G_6 \rangle,\;\langle G_4,G_7 \rangle,\;\langle G_5,G_6 \rangle,\;\langle G_5,G_8 \rangle,\;
    \langle G_6,G_8 \rangle
$$
are subalgebras of the ternary $\omega$-Lie algebra $\cubic$ isomorphic to the II-type algebra in the classification of 2-dimensional ternary $\omega$-Lie algebras. 
\end{theorem}
\begin{theorem}
The ternary $\omega$-Lie algebra ${\frak T}_0$ is a direct sum of two isomorphic three-dimensional subalgebras $\langle G_1,G_5,G_6 \rangle$ and $\langle G_2,G_3,G_4 \rangle$, that is, 
$$
{\frak T}_0=\langle G_2,G_3,G_4 \rangle\oplus \langle G_1,G_5,G_6 \rangle.
$$
The nontrivial ternary commutation relations of the algebra $\langle G_2,G_3,G_4 \rangle$ are of the form
\begin{eqnarray}
&& [G_2,G_3,G_2]=-\frac{1}{32}G_3+\frac{3i}{32}G_4,\;\;\;\;\;\;[G_3,G_2,G_3]=\frac{i}{4\sqrt{2}}G_4,\nonumber\\
&& [G_2,G_4,G_2]=\frac{3i}{32}G_3+\frac{9}{32}G_4,\;\;\;\;\;\;\;\;\;[G_4,G_2,G_4]=-\frac{3}{4\sqrt{2}}G_4,\nonumber\\
&& [G_3,G_4,G_3]=G_4,\;\;\qquad\qquad\quad\;\;\;\,\;\;\;[G_4,G_3,G_4]=G_3,\nonumber\\
&& [G_2,G_3,G_4]=\frac{i}{4\sqrt{2}}G_3-\frac{3\,\omega}{4\sqrt{2}}G_4,\;\;\,
                  [G_4,G_3,G_2]=\frac{i}{4\sqrt{2}}G_3-\frac{3\,\overline{\omega}}{4\sqrt{2}}G_4.\nonumber
\end{eqnarray}
\label{Theorem 2}
\end{theorem}
It follows from Theorem 2 that the ternary $\omega$-Lie algebra of cubic matrices of the second order is the direct sum of four two-dimensional subalgebras, that is, 
$$
\cubic=\langle G_1,G_2\rangle\oplus\langle G_3,G_4\rangle\oplus\langle G_5,G_6\rangle\oplus\langle G_7,G_8\rangle,
$$
where the first three are isomorphic to the II-type algebra in the classification of 2-dimensional ternary $\omega$-Lie algebras and the fourth one is Abelian. 
\section{Discussion}
In this paper, we develop and investigate the structure of a ternary Lie algebra at cube roots of unity. 
A ternary Lie algebra at cube roots of unity, or more briefly, a ternary $\omega$-Lie algebra, 
can be regarded as a ternary extension of the concept of a Lie algebra. 
Our approach differs from those of Filippov and Nambu in that the proposed ternary bracket is not antisymmetric but possesses $\omega$-symmetry. 
As a consequence, the ternary bracket does not vanish when two of its three elements are equal. 
However, if all three elements are equal, the ternary bracket is equal to zero. 
The physical motivation for such an extension of the concept of a Lie algebra comes from the ternary generalization of the Pauli exclusion principle proposed by Kerner. 
This generalization states that in a quantum system consisting of several particles, 
three particles with identical quantum characteristics cannot coexist, 
whereas two such particles may coexist within the same quantum system.

In this paper, we also propose a method for constructing associative ternary multiplications, 
assuming that we are given a module over an algebra. 
The elements of this module are regarded as states of a quantum system. 
In general terms, this method consists in defining on a module an algebra-valued bilinear form 
and then constructing the ternary multiplication of states as follows: 
using the bilinear form, we construct from two states of the system an operator (an element of the algebra), 
and then act with this operator on the third state. 
In this way, the ternary multiplication provides a unified structure that combines both the operators 
and their action on the states of the system.

\section*{Acknowledgments}
{We are grateful to Sergei Silvestrov, Ying Ni, Lars Hellstr\"om and German Garcia for valuable discussions during the Workshop "Exploring the World of Mathematics III" held at M\"alardalen University from September 15 to 17, 2025. We also express our gratitude to Chengming Bai and Abdenacer Makhlouf for the fruitful discussion of our research 
and its relation to 3-Lie algebras during the conference 
``Workshop on algebras and applications in mathematical physics'' 
held in June 2025 at the Chern Institute of Mathematics in Tianjin.
}
\end{document}